\documentclass[10pt]{amsart}
\usepackage{enumerate,  mathrsfs,color,amssymb}
\usepackage[utf8]{inputenc}
\usepackage[T1]{fontenc}

\usepackage{amsmath}
\usepackage{float}
\usepackage[all]{xy}
\usepackage{fancyhdr}
\usepackage{graphicx}

\usepackage{mathtools}
\usepackage{adjustbox}
\usepackage{enumitem}
\usepackage{stackrel}
\usepackage{scalerel}
\usepackage{euscript}
\usepackage{amsmath}
\usepackage{amsthm}
\usepackage{indentfirst}

\usepackage{mathpazo}
\usepackage{mathrsfs}
\usepackage{amsmath}
\usepackage{amssymb}
\usepackage{hyperref}
\usepackage{cleveref}
\usepackage{comment}
\usepackage{color}
\usepackage{tikz}
\usepackage{tikz-cd}
\usepackage{xparse}
\usetikzlibrary{matrix,arrows,decorations.pathmorphing}

\newtheorem{defn}{Definition}[section]
\newtheorem{thm}{Theorem}[section]
\newtheorem{prop}[thm]{Proposition}
\newtheorem{cor}[thm]{Corollary}
\newtheorem{lem}[thm]{Lemma}

\newtheorem{remark}[thm]{Remark}
\newtheorem{example}[thm]{Example}

\crefname{lem}{Lemma}{lemma}
\crefname{remark}{Remark}{remark}
\crefname{cor}{Corollary}{corollary}
\crefname{thm}{Theorem}{theorem}
\crefname{prop}{Proposition}{proposition}
\crefname{example}{Example}{example}
\crefname{defn}{Definition}{definition}
\crefname{notation}{Notation}{notation}
\crefname{appendix}{Appendix}{appendix}
\crefname{section}{Section}{section}

\newcommand{\TTT}{\mathcal T}
\newcommand{\NNN}{\mathcal N}

\newcommand{\OOO}{\mathcal O}
\newcommand{\PPP}{\mathcal P}

\newcommand\Zb {\mathbb{Z}}

\newcommand{\F}{\mathscr {F}}

\newcommand\CA {\EuScript{A}}
\newcommand\CB {\EuScript{B}}
\newcommand\CC {\EuScript{C}}

\newcommand\CH {\EuScript{H}}
\newcommand\CI {\EuScript{I}}

\newcommand\CK {\EuScript{K}}
\newcommand\CL {\EuScript{L}}

\newcommand\CS {\EuScript{S}}

\newcommand{\FB} {\mathfrak{B}}

\usepackage{amsfonts} 
\DeclareMathSymbol{\shortminus}{\mathbin}{AMSa}{"39}
\usepackage{mathpazo}

\DeclareMathOperator{\ob}{Ob}

\DeclareMathOperator{\Ker}{Ker}

\DeclareMathOperator{\Aut}{Aut}
\DeclareDocumentCommand{\lin}{m O{\cdot} O{\cdot}}{{}_{{#1}{ }} \langle #2, #3 \rangle}
\DeclareDocumentCommand{\rin}{m O{\cdot} O{\cdot}}{\langle #2, #3 \rangle_{{#1}{}} }
\DeclareDocumentCommand{\ot}{m}{\otimes_{{#1}{}} }
\DeclareDocumentCommand{\od}{m}{\odot_{{#1}{}} }
\DeclarePairedDelimiterX\braket[2]{\langle}{\rangle}{#1 \delimsize\vert #2}

\newcommand{\titleinfo}{ Cuntz-Nica-Pimsner algebras associated to product systems over quasi-lattice ordered groupoids}
\newcommand{\titleinfoshort}{ Cuntz-Nica-Pimsner algebras}
\newcommand{\authorinfo}{Feifei Miao, Liguang Wang, Wei Yuan}

\begin{document}

\title{\LARGE\textbf{\titleinfo}}
\author{\large\textsc{\authorinfo}}

\address{School of Mathematical Sciences, Qufu Normal University, Qufu, Shandong, 273165, China}
\email{mff100511@163.com }

\address{School of Mathematical Sciences, Qufu Normal University, Qufu, Shandong, 273165, China}
\email{wangliguang0510@163.com}

\address{AMSS, Chinese Academy of Sciences, Beijing, 100190,  China\\
and\\
School of Mathematical Sciences\\
University of Chinese Academy of Sciences, Beijing 100049, China}
\email{wyuan@math.ac.cn}

\begin{abstract}
We characterize Cuntz-Nica-Pimsner algebras for compactly aligned product systems over quasi-lattice ordered groupoids. We show that the full cross sectional $C^*$-algebras of Fell bundles of Morita equivalence bimodules are isomorphic to the related Cuntz-Nica-Pimsner algebras under certain conditions.
\end{abstract}

 \thanks{Wang was supported in part by NSF of China (No. 11871303, No. 11971463, No. 11671133) and NSF of Shandong Province (No. ZR2019MA039, No. ZR2020MA008).
 Yuan was supported in part by NSF of China (No. 11871303, No. 11871127, No. 11971463).
 }
\date{}
\maketitle

\baselineskip=18pt \vskip12pt

\section{Introduction}

 A class of $C^*$-algebras later called Cuntz-Krieger algebras were introduced by Cuntz and Krieger in \cite{MR467330} and \cite{MR561974}.  Since then, there have been many  generalizations of  Cuntz-Krieger algebras. In \cite{MR1426840},  Pimsner studied  a class of  $C^*$-algebras associated to C$^*$-correspondences which also generalized the Cuntz-Krieger algebras. Pimsner's work was refined by Katsura \cite{MR2102572}. The C$^{\ast}$-algebras constructed by Katsura were later called  Cuntz-Pimsner algebras.

From another perspective, Fowler generalized Pimsner's work and studied $C^*$-algebras for regular product systems of C$^{\ast}$-correspondences over directed quasi-lattice ordered groups \cite{FowlerN2002}.
Based on the work of Nica \cite{MR1241114} and of Laca and Raeburn \cite{MR1402771}, Folwer introduced the notions of Nica covariant Toeplitz representations of product systems and Cuntz-Pimsner covariant Toeplitz representations of compactly aligned product systems. Moreover, he described Nica-Toeplitz algebras which  are universal for Nica covariant Toeplitz representations, and  Cuntz-Pimsner algebras which are universal for Cuntz-Pimsner covariant Toeplitz representations respectively.
However, it was unclear what the right analogue of the Cuntz algebras for product systems  of non-injective $C^*$-correspondences should be and this was left open for quite some time.  
 Eventually, Sims and Yeend  described Cuntz-Nica-Pimsner algebras  for some  product systems  over quasi-lattice ordered groups under certain conditions \cite{Sims2010}.  Cuntz-Nica-Pimsner algebras are universal for Cuntz-Nica-Pimsner covariant representations which are Nica covariant in the sense of Fowler, and satisfy Cuntz-Pimsner covariant  conditions which are  generalizations of Fowler's Cuntz-Pimsner covariant conditions.

There is also another  version of Cuntz-Pimsner algebras which are co-universal. This class of C$^{\ast}$-algebras were introduced by  Carlsen, Larsen, Sims  and Vittadello in \cite{MR2837016} and are co-universal for injective, Nica covariant, gauge-compatible representations of product systems under certain conditions.   It was shown in \cite{MR2837016}  that the co-universal Cuntz-Pimsner algebras and the Cuntz-Nica-Pimsner algebras of some product systems are generalizations of  reduced and full  crossed products of $C^*$-algebras by groups respectively.  In further work, Dor-On, Kakariadis, Katsoulis, Laca, and Li \cite{Li} generalized  Dor-On and Katsoulis' results \cite{MR4053621}  and  showed that there exist  co-universal algebras of compactly aligned product systems over group-embeddable right LCM semigroups.
Motivated by the above results, we introduced  the notion of product systems over quasi-lattice ordered groupoids in \cite{Feifei}. And we  showed that the co-universal algebras  of  compactly aligned product systems over quasi-lattice ordered groupoids do exist.

In this paper, we will   describe  Cuntz-Nica-Pimsner algebras for some compactly aligned product systems over quasi-lattice ordered groupoids which are universal for injective Cuntz-Nica-Pimsner covariant representations.

This paper is organized as follows. In section 2, we recall the terminologies and basic results which will be used later. In section 3, we introduce the Cuntz-Pimsner condition for Toeplitz representations and study the Cuntz-Nica-Pimsner representations for compactly aligned product systems over quasi-lattice ordered groupoids. In section 4, we show that 
 the full cross sectional $C^*$-algebra of the Fell bundle is isomorphic to the Cuntz-Nica-Pimsner algebra of the associated product system
when the fibers of a Fell bundle over a quasi-lattice ordered groupoid  are all Morita equivalence bimodules.

\section{Preliminaries}

\subsection{$C^{\ast}$-correspondences}

Let $\CB$ be a C$^*$-algebra and $X$ be a right Hilbert $\CB$-module (see \cite{L95}). We use $\CL(X)$ (resp. $\CK(X)$) to denote the C$^*$-algebra of all adjointable (resp. compact) operators on $X$.



 \begin{defn}
For $C^{\ast}$-algebras $\CA$ and $\CB$, a $\CA$-$\CB$ $C^{\ast}$-correspondence is a right Hilbert $\CB$-module $X$ together with a left $\CA$-action induced by a non-degenerate $*$-homomorphism $\phi : \CA \to \CL(X)$, i.e. $ \CA \cdot X := \overline{\rm{span}}\{ \phi(a)\xi: a \in \CA, \xi \in X\} = X$.
   \end{defn}

To simplify notation, we will write $\phi(a) \xi $ as $a \xi$ in the following.

    \begin{defn}
    A bimodule map $F:X \to Y$ between two $\CA$-$\CB$ C$^*$-correspondences $X$ and $Y$  is called a \textit{C$^*$-correspondence map} if $$\rin{\CB}[\xi][\beta] = \rin{\CB}[F(\xi)][F(\beta)]$$ for all $\xi$, $\beta \in X$. A $C^*$-correspondence map is called a \textit{C$^*$-correspondence isomorphism} if it is a unitary.
    \end{defn}

\begin{remark}
    A C$^*$-correspondence map  between two $\CA$-$\CB$ C$^*$-correspondences $X$ and $Y$ is not adjointable in general. For example, let $\Omega$ be a compact Hausdorff space and $\Omega_0$ be a non-empty closed subset of $\Omega$ whose complement is dense in $\Omega$. Then $$X := \{ f \in C(\Omega): f(\Omega_0) = 0\}$$ is a non-unital C$^*$-subalgebra of $C(\Omega)$. It is clear that $X$ and $C(\Omega)$ are both $C(\Omega)$-$C(\Omega)$ C$^*$-correspondences with the canonical structure. And the inclusion map $F: X \to Y$ is a C$^*$-correspondence map. But $F$ is not adjointable since $F^*$, if exists, must map $1_{C(\Omega)}$ to the unit of $X$.
\end{remark}

Let $X$ be a $\CA$-$\CB$ C$^{\ast}$-correspondence and $Y$ be a $\CB$-$\CC$ $C^*$-correspondence, where $\CA$, $\CB$, and $\CC$ are C$^{\ast}$-algebras. The $\CB$-balanced tensor product $X \od{\CB} Y$ together with right $\CC$-action
 \begin{align*}
 (\xi \odot \eta) \cdot c = \xi \odot \eta \cdot c
 \end{align*}
 and $\CC$-valued inner product
  \begin{align*}
        \rin{\CC}[\xi_1 \odot \eta_1][\xi_2 \odot \eta_2]:=  \rin{\CC}[\eta_1][\rin{\CB}[\xi_1][\xi_2]\eta_2]
    \end{align*}
 for $\xi,  \xi_1, \xi_2 \in X$, $\eta, \eta_1, \eta_2 \in Y$,
 is a pre-Hilbert right $\CC$-module. We use $X \ot{\CB} Y$ to denote the completion of $X \od{\CB} Y$. Then $X \ot{\CB} Y$ is a $\CA$-$\CC$ C$^*$-correspondence.

\begin{defn} A $\CA$-$\CB$ $C^*$-correspondence $X$ is called  a  Hilbert $C^*$ $\CA$-$\CB$-bimodule if $X$ equipped with a left $\CA$-valued inner product $\lin{\CA}$ such that the left $\CA$-module $X$ together with $\lin{\CA}$ is a left Hilbert $\CA$-module and
\begin{align*}
\lin{\CA}[\xi][\eta] \zeta = \xi \rin{\CB}[\eta][\zeta], \quad \forall \xi, \eta, \zeta \in X.
\end{align*}
If, moreover, $\lin{\CA}[X][X]$ and $\rin{\CB}[X][X]$ are dense in $\CA$ and $\CB$ respectively, then $X$ is called a Morita equivalence bimodule.
\end{defn}


\subsection{Product systems over quasi-lattice ordered groupoids}

Let $G$ be a small category. We use $\ob(G)$ to denote the set of objects in $G$ and $G(x,y)$ to denote the set of morphisms in $G$ with source $x$ and target $y$. Let $G(-, -)$ be the set of morphisms in $G$. For every $x, y \in \ob(G)$, we use $G(x, -)$ (resp. $G(-, y)$) to denote the set of morphisms in $G$ with source $x$ (resp. target $y$). For $f \in G(-, -)$, we use $s(f)$ (resp. $t(f)$) to denote the source (resp. target) of $f$. For $x \in \ob(G)$, we use $1_x$ to denote the identity morphism of $x$. 
A discrete groupoid is a category in which every morphism is invertible (We refer the reader to \cite{RS} for a basic introduction of category theory).

\begin{defn}[\cite{Feifei}]
 A quasi-lattice ordered groupoid is a pair $(G, P)$, where $G$ is a discrete groupoid and $P$ is a wide subcategory of $G$, i.e., $\ob(P) = \ob(G)$, such that
\begin{enumerate}
    \item for every $x, y \in \ob(P)$,
   \begin{align*}
     P(x,y) \cap P^{-1}(x,y) =
     \begin{cases}
         \emptyset,& \mbox{if $x \neq y$;}\\
          1_x, & \mbox{if $x = y$,}
      \end{cases}
    \end{align*}
    where $P^{-1}(x,y) := \{f^{-1}: f \in P(y,x)\}$;

    \item under the partial order $f \leq g \iff f ^{-1}g \in P$, $f, g \in G(-, -)$,  every finite set of morphisms in $G$ with an upper bound in $P$ has a least upper bound in $P$.
\end{enumerate}
\end{defn}

\begin{example}[\cite{Feifei}]Let $(G', P')$ be a quasi-lattice ordered group.
 \begin{enumerate}
     \item Note that $G'$ is a groupoid with only one object and $(G', P')$ is a quasi-lattice ordered groupoid.
     \item Let $X$ be a $G'$-set, $G'$ acts on $X$. The transformation groupoid $G$ is the category with $\ob(G) = X$ and $G(x, y) = \{g\in G': g \cdot x =y\}$ (see example 8.1.15 in \cite{Sims20101}). Let $P$ be the wide subcategory of $G$ such that $P(x, y) = \{ g \in P': g \cdot x =y\}$.
Recall that $P' \cap P'^{-1} = \{ e\}$, where $e$ is the unit of $G'$.
 Then $(G, P)$ is a quasi-lattice ordered groupoid.
     \end{enumerate}
     \end{example}

For the rest of this paper, we assume that  $(G, P)$ is a quasi-lattice ordered groupoid. For $f, g \in G(-, -)$, if they have a common upper bound in $P$, we use $f\vee g$ or $\vee\{f,g\}$ to denote their least upper bound.  Otherwise, we  write $f\lor g=\infty$.





\begin{defn}\label{def:_product_sys}
A product system over (G, P) is a triple $$(\{\CA_{x}\}_{x \in Ob(P)}, \{X_{f}\}_{f \in P(-,-)}, \{m_{g, f}\}_{\{(g, f) \in P(-, -): t(f) = s(g)\}}),$$ where $\CA_x$ is a  C$^*$-algebra, $X_{f}$ is a $\CA_{t(f)}$-$\CA_{s(f)}$ C$^*$-correspondence, and $m_{g, f}$ is a $C^*$-correspondence isomorphism from $X_g \ot{\CA_{t(f)}} X_f$ to $X_{gf}$, subject to the following axioms:
\begin{enumerate}
    \item $X_{1_x} = \CA_x$ for every $x \in \ob(P)$;
    \item $m_{1_x, f}(a \otimes_{\CA_x} \xi) = a \xi$ and $m_{f, 1_w}(\xi \otimes_{\CA_w} b) = \xi b$ for $f \in P(w,x)$, $a \in \CA_x$, $b \in \CA_w$, $\xi \in X_f$;
    \item for every $f \in P(w, x)$, $g \in P(x,y)$, $h \in P(y,z)$, the following diagram

    \begin{align*}
        \xymatrix @R=0.2in @C=0.6in{
            X_h \ot{\CA_y} X_g\ot{\CA_{x}} X_f \ar[d]^{I_{X_h} \ot{\CA_y} m_{g,f}}
            \ar[r]^-{m_{h,g} \ot{\CA_{x}} I_{X_f}} & X_{hg} \ot{\CA_{x}} X_f \ar[d]^{m_{hg,f}}\\
            X_h \ot{\CA_y}  X_{gf} \ar[r]^{m_{h,gf}} & X_{hgf}
        }
    \end{align*}
 commutes.

\end{enumerate}
\end{defn}

In the following, we will write$(\{\CA_{x}\}_{x \in Ob(P)}, \{X_{f}\}_{f \in P(-,-)}, \{m_{g, f}\})$ as  $(\{\CA_{x}\}, \{X_{f}\})$ and write  $m_{g, f}(\xi_g \otimes_{\CA_s(f)} \eta_f)$ as $\xi_g \eta_f$.

For $f, g \in P(-, -)$, there is a $*$-homomorphism  $l^g_f : \CL(X_f) \to \CL(X_g)$  defined by
\begin{align*}
l^{g}_f(A) :=\begin{cases}
     m_{f, f^{-1}g}(A \otimes_{\CA_{s(f)}} I) m^*_{f, f^{-1}g}, & \text{ if} \quad f \leq g; \\
      0, & \text{otherwise},
\end{cases}
\end{align*}
where $A \in \CL(X_f)$. By identifying $\CK(X_{1_{t(g)}})$ with $\CA_{t(g)}$, we have $ l^g_{1_{t(g)}}(a)(\xi_g) = a \xi_g$ by condition (2) in \cref{def:_product_sys},  for $\xi_g \in X_g$, $a \in \CA_{t(g)}$.

\begin{defn}
A product system $X = (\{\CA_{x}\}, \{X_{f}\})$ over $(G, P)$ is called \textit{compactly aligned} if $S \vee T := l^{f\vee g}_f(S) l^{f \vee g}_g(T)$ is compact, i.e., $S \vee T  \in \CK(X_{f \vee g})$, for every $S \in \CK(X_f)$, $T \in \CK(X_g)$, and $f, g \in P(-, -)$ such that $f \vee g < \infty$.
 \end{defn}


\subsection{Representations of a product system over a quasi-lattice ordered groupoid}\


\begin{defn}\label{def:iso_rep_comp_alig_prod_system}
    Let $(\{\CA_{x}\}, \{X_{f}\})$ be a product system over $(G, P)$. A  Toeplitz \textit{representation} of $(\{\CA_{x}\}, \{X_{f}\})$ is a family of linear maps $\{\psi_f : X_f \to \CB\}_{f \in P(-, -)}$, where $\CB$ is a C$^*$-algebra, such that
    \begin{enumerate}
        \item $\psi_{1_x} : \CA_x \to \CB$ is a $*$-homomorphism for every $x \in \ob(P)$;
        \item $\psi_g(\beta) \psi_f(\xi) = \psi_{gf}(\beta \xi)$ for every $\beta \in X_g$, $\xi \in X_f$, $g$, $f \in P(-, -)$, $s(g) = t(f)$;
        \item $\psi_f(\eta)^*\psi_f(\xi) = \psi_{1_{s(f)}} (\rin{\CA_{s(f)}}[\eta][\xi])$ for  $\xi, \eta \in X_f$, $f \in P(-,-)$.
    \end{enumerate}
\end{defn}

Let $\psi = \{\psi_f : X_f \to \CB \}_{f \in P(-, -)}$ be a Toeplitz representation of $(\{\CA_{x}\}, \{X_{f}\})$. The condition (3) in \cref{def:iso_rep_comp_alig_prod_system} implies that $\psi_f$ is a contraction, for every $ f \in P(-, -)$. And $\psi_f$ is isometric if $\psi_{1_{s(f)}}$ is injective. We say $\psi$ is injective if $\psi_{1_x}$ is injective for all $x \in \ob(P)$.

From now on, we always assume that $X = (\{\CA_{x}\}, \{X_{f}\})$ is a compactly aligned product system over $(G, P)$ unless otherwise specified. For a Toeplitz representation $\psi = \{ \psi_f: X_f\to \CB\}_{f \in P(-, -)}$ of $X$, there are $*$-homomorphisms $\psi^{(f)} : \CK(X_f) \to \CB$ such that $\psi^{(f)}(\theta_{\xi, \eta}) = \psi_f(\xi) \psi_f(\eta)^*$, where $\theta_{\xi, \eta}(-)$ is the rank one operator $\xi \rin{\CA_{s(f)}}[\eta][-]$ in $\CK(X_f)$(see \cite[p. 202]{MR1426840},  \cite[Lemma 2.2]{MR1658088} and \cite[Remark 1.7]{MR1722197}). In the following, we adopt the convention that $\psi^\infty \equiv 0$.

\begin{defn}
A Toeplitz representation $\psi= \{ \psi_f\}_{f \in P(-, -)}$ of $(\{\CA_{x}\}, \{X_{f}\})$ is called \textit{Nica covariant} if
\begin{align*}
   \psi^{(f)}(S)\psi^{(g)}(T) =  \psi^{(f \vee g)}(S \vee T),
\end{align*}
for every $f, g \in P(-, -)$, $S \in \CK(X_f)$ and $T \in \CK(X_g)$.
\end{defn}

\begin{remark}[see \cite{Feifei}]\label{rem:_some_p_of_N}
Let $\psi$ be a Nica covariant Toeplitz representation of $(\{ \CA_x\}, \{ X_f\})$. For $f , g \in P(-, -)$, $\xi_f \in X_f$, $\eta_g \in X_g$,
\begin{enumerate}
\item if  $s(f) \neq t(g)$, then  $\psi_f (\xi_f) \psi_g(\eta_g) = 0$;

\item if $f \vee g < \infty $, then $\psi_f(\xi_f)^*\psi_g(\eta_g) \in   \overline{  \rm{span}}\{ \psi_{f^{-1}(f \vee g)}(\xi') \psi_{g^{-1}(f\vee g)}(\eta')^* : \xi' \in X_{f^{-1}(f \vee g)}, \eta' \in X_{g^{-1}(f \vee g)}\}$. Otherwise, $\psi_f(\xi_f)^*\psi_g(\eta_g) = 0$.
\end{enumerate}
\end{remark}

Let $\psi = \{ \psi_f\}_{f \in P(-, -)}$ be a Nica covariant Toeplitz representation of $X$. In the following, we use $C^*(\psi)$  to denote the $C^*$-algebra generated by the image of $\psi$.
It follows from  \cref{rem:_some_p_of_N} that
\begin{align}\label{Nica_covariant_eq}
C^*(\psi) = \overline{\mbox{span}}\{ \psi_f(\xi_f) \psi_g(\eta_g)^* : f, g \in P(-, -), \xi_f \in X_f, \eta_g \in X_g \}.
\end{align}

\begin{example}[The Fock representation] \label{Sub:_fock}
Let $X = (\{\CA_{x}\}, \{X_{f}\})$ be a compactly aligned product system over $(G, P)$. We use $C_c(X)$ to denote the set of maps $$\xi: P(-,-) \to \cup_{f \in P(-,-)} X_f$$ with finite support such that  $\xi(f) \in X_f$. Similarly, $C_c(\{\CA_x\}_{x \in \ob(P)})$ is the set of maps $\beta: \ob(P) \to \cup_{x \in \ob(P)} \CA_x$ with finite support such that $\beta(x) \in \CA_x$.
Let $C_0(\{\CA_x\})$ be the completion of $C_c(\{\CA_x\})$ with respect to the supremum norm. Note that $C_0(\{\CA_x\})$ is a C$^*$-algebra, and $C_c(X)$ together with the right $C_0(\{\CA_x\})$-action and the $C_0(\{\CA_x\})$-valued inner product
    \begin{align*}
        [\xi \cdot \beta](f) := \xi(f) \cdot \beta(s(f)),\quad
        \rin{C_0(\{\CA_x\})}[\xi][\eta](x) :=\sum_{f \in P(x, -)} \rin{\CA_{x}}[\xi(f)][\eta(f)],
    \end{align*}
is a pre-Hilbert right module over $C_0(\{\CA_x\})$. The Fock space $\F(X)$ is the completion of $C_c(X)$. The \textit{Fock representation}  $\{\eta_g \mapsto L_{\eta_g}: X_g \to \CL(\F(X))\}_{g \in P(-,-)}$ of $X$ is  defined by
\begin{align*}
    (L_{\eta_g}\xi)(f) =
    \begin{cases}
        \eta_g \xi(g^{-1}f), & g \leq f;\\
        0, & \mbox{otherwise}.
    \end{cases}
\end{align*}
The Fock representation is a Nica covariant Toeplitz representation of $X$. And the \textit{Fock algebra} $\TTT^r_X$ is the $C^*$-algebra generated by the image of the Fock representation of $X$ (see \cite{Feifei} for more details).
\end{example}

The Nica-Toeplitz algebra for $X$ is a C$^*$-algebra $\NNN\TTT_X$ equipped with a Nica covariant Toeplitz representation $i = \{ i_f: X_f \to \NNN\TTT_X\}_{f \in P(-, -)}$ such that
\begin{enumerate}
    \item $\NNN\TTT_X = C^*(i)$,
    \item $i$ is universal in the sense that
for any other  Nica covariant Toeplitz representation $\psi : X \to \CB$, there is a unique $*$-homomorphism $\psi' : \NNN\TTT_X \to \CB$ rendering the following diagram
   \begin{align*}
\xymatrix @R=0.2in {
   &  X \ar[ld]_{\psi} \ar[rd]^-{i} &  \\
  \CB & &\NNN\TTT_X \ar[ll]_-{\psi'}
   }
\end{align*}
commutative.
\end{enumerate}

\subsection{The full cross sectional $C^*$-algebra  $C^*(\FB)$ of a Fell bundle $\FB$ }\

Let $\{\FB_f\}_{f \in G(-, -)}$ be a Banach bundle over discrete groupoid $G$ (see \cite[Definition 13.4]{MR936628}). We  write  $\FB_{1_x}$ as  $\FB_x$ for convenience. A multiplication on $\{\FB_f\}_{f \in G(-, -)}$ is a family of continuous bilinear maps
\begin{align*}
  \{\FB_f \times \FB_g \to \FB_{fg}\}_{(f,g) \in G(-, -)^2, s(f) = t(g)}
\end{align*}
such that
\begin{enumerate}
    \item $(\xi_f \eta_g) \zeta_h = \xi_f (\eta_g \zeta_h)$ whenever the multiplication is defined;
    \item $\| \xi_f \eta_g\| \leq \|\xi_f\| \| \eta_g\|$ for all $\xi_f \in \FB_f$, $\eta_g \in \FB_g$, $f, g \in G(-, -)$, $s(f) = t(g)$.
\end{enumerate}
An involution on $\{\FB_f\}_{f \in G(-, -)}$ is a family of continuous involutive conjugate linear maps
\begin{align*}
    \xi_f \mapsto \xi_f^* : \FB_{f} \to \FB_{f^{-1}}, \quad f \in G(-, -), \xi_f \in \FB_f.
\end{align*}

\begin{defn}[\cite{Kumjian1998}]\label{def:_the_fell_bundle_groupoid}
A Fell bundle over a discrete groupoid $G$ is a Banach bundle $\{\FB_g\}_{g \in G(-, -)}$  equipped with a multiplication and an involution such that  for $f, g \in G(-, -)$, $\xi_f \in \FB_f$, $\eta_g \in \FB_g$ with $s(f) = t(g)$,
\begin{align*}
    (\xi_f \eta_g)^* = \eta^*_g \xi^*_f, \quad \| \xi_f^* \xi_f\| = \| \xi_f\|^2,
\end{align*}
and $\xi_f^* \xi_f$ is a positive element of the C$^*$-algebra $\FB_{s(f)}$.
\end{defn}

For the rest of this subsection, we assume that $\FB = \{\FB_f\}_{f \in G(-, -)}$ is a Fell bundle over a discrete groupoid $G$.
\begin{remark}
    For each $f \in G(-,-)$,  $\FB_f$ together  with inner products
    \begin{align*}
    \rin{\FB_{s(f)}}[\xi_f][\eta_f] := \xi^*_f \eta_f, \quad \lin{\FB_{t(f)}}[\xi_f][\eta_f] := \xi_f \eta^*_f
    \end{align*}
     is a Hilbert $C^*$ $\FB_{t(f)}$-$\FB_{s(f)}$-bimodule. And $\FB^*_f = \FB_{f^{-1}}$.
    \end{remark}

\begin{defn}[\cite{Feifei}]
 A representation of $\FB $ in a $C^*$-algebra $\CC$ is a collection of linear maps $\pi = \{ \pi_g: \FB_g \to \CC\}_{g \in G(-, -)}$ such that
\begin{enumerate}
\item $\pi_f(\xi_f) \pi_g(\eta_g) =  \delta_{s(f), t(g)}\pi_{fg}(\xi_f \eta_g) $;
\item $\pi_g(\eta_g)^* = \pi_{g^{-1}}(\eta^*_g)$,
\end{enumerate}
 for all$f, g \in G(-, -)$,  $\xi_f \in \FB_f$, $\eta_g \in \FB_g$.
\end{defn}

Recall that $C_c(\FB)$ is the set of  finitely supported sections of $\FB$, i.e., a family of maps $\xi : G(-, -) \to \cup_{g \in G(-, -)} \FB_g$ with finite support such that $ \xi(g) \in \FB_g$.
Then $C_c(\FB)$ together with multiplication and involution
\begin{align*}
\xi \eta (h) := \sum_{h = fg} \xi(f) \eta(g),\quad  \xi^*(f) := \xi(f^{-1})^*
\end{align*}
 is a $*$-algebra (see \cite{Feifei} for more detials).
\begin{defn}[\cite{Feifei}]
The full cross sectional $C^*$-algebra  $C^*(\FB)$ of $\FB$ is the completion of $C_c(\FB)$ with respect to the norm
\begin{align*}
\| \xi \|_u := sup_{\pi}\{ \| \pi(\xi)\|: \pi \ \ \textit{is the $*$-representation of $C_c(\FB)$}\},
\end{align*}
where $\xi \in C_c(\FB)$.
\end{defn}

 The set of representations of $\FB$ in C$^*$-algebras corresponds to the set of  $*$-homomorphisms from $C_c(\FB)$ into C$^*$-algebras one-to-one.
Then $C^*(\FB)$ is a $C^*$-algebra generated by the image of the universal  representation $\pi^u = \{ \pi^u_g\}$ of $\FB$  in the sense that  if $\pi = \{ \pi_g\}_{g \in G(-, -)}$ is a representation of $\FB$ in  a  $C^*$-algebra $\CB$, then there exists a unique $*$-homomorphism $\widehat{\pi} : C^*(\FB) \to \CB$ such that $\widehat{\pi} \circ \pi^u_g(\xi_g) = \pi_g(\xi_g)$ for all $g \in G(-, -)$ and $\xi_g \in \FB_g$.

\section{Cuntz-Nica-Pimsner algebras of product systems over quasi-lattice ordered groupoids}
Let $X= (\{ \CA_x\}, \{X_f\})$ be a compactly aligned product system over $(G, P)$. For $g \in P(x,y)$, let
\begin{align*}
I_g:= \begin{cases}
    \CA_y & x=y \mbox{ and } g = 1_y;\\
    \cap_{ 1_y < f\leqslant g} \Ker \phi_f & x=y \mbox{ and } g \neq 1_y; \\
    0 &  x \neq y,
\end{cases}
\end{align*}
where $\phi_f$ is the left action of $\CA_{y}$ on $X_f$. We use $C^g_c(X)$ to denote the set of maps $$\xi: \{ f \in P(x, y) : f \leq g \} \to \cup_{f \in P(x, y), f \leq g} X_f I_{f^{-1} g}$$ with finite support such that $\xi(f) \in  X_f I_{f^{-1} g}$. Note that $C^g_c(X)$ together with the right $\CA_x$-action and right $\CA_x$-valued inner product
\begin{align*}
[\xi \cdot a_x] (f):= \xi(f) a_x  \ \  \mbox{and} \ \ \rin{\CA_x}[\xi][\eta] = \sum_{f \in P(x, y), f \leq g}\rin{\CA_x}[\xi(f)][\eta(f)],
\end{align*}
for all $\xi, \eta \in C^g_c(X)$ and $a_x \in \CA_x$,  is a pre-Hilbert right $\CA_x$-module.
Let $\widetilde{X}_g$ be the completion of $C^g_c(X)$.
We will use $\widetilde{\phi}_g$ to denote the left action of $\CA_y$ on $\widetilde{X}_g$ induced by $\phi_g$. We say $X$ is $\widetilde{\phi}$-injective if every $\widetilde{\phi}_g$ is injective.

For $f, g \in P(-,y)$,  let  $\widetilde{l}_f^g : \CL(X_f) \to \CL(\widetilde{X}_g)$ be the $*$-homomorphism defined by
  \begin{align*}
  \widetilde{l}^g_f (A)(\xi)(h) := l^h_f(A)(\xi(h)),
 \end{align*}
 for $A \in L(X_f)$, $\xi \in \widetilde{X}_g$, $h \in P(s(g), y)$ and $h \leqslant g$. In particular, $\widetilde{l}_{1_y}^g (a) (\xi) = \widetilde{\phi}_g(a) \xi$, for all $a \in \CA_{y}$.

Similar to the definition 3.8 in \cite{Sims2010},  we say a statement $\PPP(g)$ is true for large $g$ if for every $f \in P(-, -)$, there exists $h \in P(-, -)$, $h \geqslant f$ such that $\PPP(g)$ is true whenever $g\geq h$.

\begin{defn}
Assume that $X$ is $\widetilde{\phi}$-injective.
A Toeplitz representation  $\psi = \{ \psi_f \}_{f \in P(-, -)}$ of $X$  is Cuntz-Pimsner covariant if $\sum_{f \in \F}\psi^{(f)}(T_f) = 0$ for every $y \in \ob(P)$, every finite subset $\F$ of  $P(-, y)$,  and every choice of compact operators $\{ T_f \in \CK(X_f) : f \in \F\}$ such that $\sum_{f \in \F}\widetilde{l}^g_f(T_f) = 0$ for large $g$.

A Nica covariant Toeplitz representation is called a Cuntz-Nica-Pimsner covariant (or CNP-covariant) representation if it is Cuntz-Pimsner covariant.
\end{defn}

A nonempty set $\CS \subseteq P(-, y)$ is said to be bounded if there exists $g \in P(-, y)$ such that $k \leqslant g$ for every $k \in \CS$. An element $f \in \CS$ is called maximal (resp. minimal) if $f \nleqslant h$ (resp. $h \nleq f$) for every $h \in \CS \backslash \{f \}$.

We next show that $X$ is $\widetilde{\phi}$-injective under some conditions.

\begin{prop}\label{lem: J_X_is_injective }
\begin{enumerate}
\item If $\phi_f$ is injective for all $f \in P(-, -)$, then $X$ is $\widetilde{\phi}$-injective.
\item   Suppose every nonempty bounded subset of $P(x, y)$ contains a maximal element and for every  $k \in P(x, y)$, $\{ f \in P(x, y): f \leq k\}$ has a minimal element $f'$
such that $\phi_{f'}$ is injective. Then $X$ is $\widetilde{\phi}$-injective.
\end{enumerate}
\end{prop}

\begin{proof}
(1)  If $\phi_g$ is injective for every $g \in P(-, -)$, then we can identify $\widetilde{X}_g$ with $ X_g$. So $\widetilde{\phi}_g$ is injective.

(2)   Let $g \in P(x, y)$ and $0 \neq a \in \CA_y$. It is enough to show that there exist some $f \leqslant g$, $f \in P(x, y)$, such that $\phi_f (a) |_{X_f \cdot I_{f^{-1}g}}\neq 0$.
Let $$\CS = \{ f \in P(x, y) : f \leqslant g, \phi_f(a)\neq 0\}.$$ Since  $\{ f \in P(x, y) : f \leqslant g\}$  has a  minimal element  $f'$ such that $\phi_{f'}$ is injective,  $\CS$ is nonempty.  Let $f_0$ be a maximal element of $\CS$. Then there is a $\xi \in X_{f_0}$ such that $\phi_{f_0}(a)\xi \neq 0$. If $1_x < h \leqslant f_0^{-1}g$, we have $ \phi_{f_0h}(a)=0$.
For $\eta \in X_h$, we have
\begin{align*}
\rin{\CA_{s(h)}}[\eta][\rin{\CA_{x}}[\phi_{f_0}(a)\xi][\phi_{f_0}(a)\xi] \eta] = \rin{\CA_{s(h)}}[(\phi_{f_0}(a)\xi) \eta][(\phi_{f_0}(a) \xi )\eta]  = 0.
\end{align*}
So $\rin{\CA_{x}}[\phi_{f_0}(a)\xi][\phi_{f_0}(a) \xi] \in \ker \phi_h$.
It follows that $\rin{\CA_{x}}[\phi_{f_0}(a)\xi][\phi_{f_0}(a) \xi] \in I_{f_0^{-1}g}$.
Recall that if $X$ is a right Hilbert $\CA$-module and $I$ is an ideal of $\CA$, then
$$X\cdot I = \{ \xi \in X: \rin{\CA}[\xi][\xi] \in I\}$$
(see \cite{Sims2010}).
Then $\phi_{f_0}(a)\xi \in X_{f_0} I_{f_0^{-1}g}$.
Suppose $\{ e_i\}_{i \in \CI}$ is an increasing positive approximate identity for $I_{f_0^{-1}g}$.
Then
\begin{align*}
0 \neq \phi_{f_0}(a) \xi  = \lim_{i \in \CI} \phi_{f_0}(a) \xi \cdot e_{i}.
\end{align*}
Then there exists $i_0$ such that $ \phi_{f_0}(a) \xi \cdot e_{i_0} \neq 0$.  We have $\phi_{f_0}(a) |_{X_{f_0} I_{f_0^{-1}g}} \neq 0$ since $ \xi \cdot e_{i_0} \in X_{f_0} I_{f_0^{-1} g}$. This completes the proof.
\end{proof}

For a Nica covariant Toeplitz representation $\psi=\{ \psi_f\}_{f \in P(-, -)}$ of product system $X = ( \{ \CA_x\}, \{ X_f \})$, let $I_{C^*(\psi)} \subset C^*(\psi)$ be the closed ideal generated by $\cup_{y \in \ob(P)} I_{ C^*(\psi), y}$, where
\begin{align*}
I_{C^*(\psi), y}= &\{ \sum_{f \in \F} \psi^{(f)} (T_f) : \textit{$\F$ is a finite set in $P(-, y)$}, \\
 & \text{$T_f \in \CK(X_f)$  and $\sum_{f \in \F} \widetilde{l}^g_f (T_f) = 0$ for large $g \in P(-, y)$}\}.
\end{align*}
Let $q_{C^*(\psi)}$ be the quotient map from $ C^*(\psi)$ to $C^*(\psi) / I_{C^*(\psi)}$.

\begin{lem}\label{lem:CNP_con}
With  notations as above.
$$ q_{C^*(\psi)} \circ \psi:= \{ q_{C^*(\psi)} \circ \psi_f:  X_f \to C^*(\psi) / I_{C^*(\psi)}\}_{f \in P(-, -)}$$
is a CNP-covariant Toeplitz representation.
\end{lem}
\begin{proof}
Since  $\psi$ is  a Nica covariant Toeplitz representation of $X$ and $q_{C^*(\psi)}$ is a $*$-homomorphism,  $q_{C^*(\psi)} \circ \psi$ is a Nica covariant Toeplitz representation of $X$  and hence
$q_{C^*(\psi)} \circ \psi$ is  CNP-covariant.
\end{proof}

\begin{lem}\label{pro:_faith_ful_property}
If  $X$ is $\widetilde{\phi}$-injective, then $ q_{\TTT^r_X} \circ L $ is injective, where $L$ is the Fock representation of $X$ and $\TTT^r_X = C^*(L)$ (see  \cref{Sub:_fock}).
\end{lem}

\begin{proof}
It suffices to show that for every $y \in \ob(P)$,  we have  $I_{\TTT^r_X} \cap L_{\CA_y} = \{0\}$.  Note that $\widetilde{X}_g$ is a submodule of $\F(X)$, and $L_a\xi = \widetilde{\phi}_g(a)\xi$ for $a \in \CA_y$, $\xi \in \widetilde{ X}_g$, $g \in P(-, y)$.  Since
$\widetilde{\phi}_g$ is injective,  $\| L_a |_{\widetilde{X}_g} \|= \|\widetilde{\phi}_g(a)\| = \|a \|$.

We use $\overline{I}_{\TTT^r_X, z}$ to denote the closure of $I_{\TTT^r_X, z}$. Let $S \in \overline{I}_{\TTT^r_X,z}$ and $\xi_f\in X_f$, $ \eta_k \in X_k$, $ \xi'_{f'} \in X_{f'}$, $\eta'_{k'} \in X_{k'}$. We claim that for every $\epsilon \geq0$, and $g_0 \in P(-, y)$, there exists $\gamma_0 \geqslant g_0$ such that $\|  L_{\xi_f} L_{\eta_k}^* S L_{\xi'_{f'}} L_{\eta'_{k'}}^*|_{\widetilde{X}_h}\| \leq \epsilon$ for every $h \geqslant \gamma_0$.

If $k' \vee g_0 = \infty$, then for $h \geqslant g_0$, we have $h \ngeqslant k'$. Therefore $ L_{\eta'_{k'}}^*|_{\widetilde{X}_h} = 0$. And $$ L_{\xi_f} L_{\eta_k}^* S L_{\xi'_{f'}} L_{\eta'_{k'}}^*|_{\widetilde{X}_h} =0, \forall h \geqslant g_0.$$ Assume now that $k' \vee g_0 <\infty$. If $f' \notin P(s(k'), z)$, then $L_{\xi_f} L_{\eta_k}^* S L_{\xi'_{f'}} L_{\eta'_{k'}}^* |_{\widetilde{X}_h} = 0$ for every $h \in P(-, -)$. If $ f' \in P(s(k'), z)$, then there's $\gamma' \geqslant f' k'^{-1} (k' \vee g_0)$ such that
  \begin{align}\label{equ:_ineq}
  \| S |_{\widetilde{X}_{h'}}\| \leq \epsilon / (\| \xi_f \| \| \eta_k\| \| \xi'_{f'} \| \| \eta'_{k'}\|)
  \end{align} for every $h' \geqslant \gamma'$, since $S \in \overline{I}_{\TTT^r_X, z}$. Let $\gamma_0 := k' f'^{-1}\gamma' (\geqslant g_0)$. Then for $h \geqslant \gamma_0$ we have $f' k'^{-1} h \geqslant \gamma'$. Thus \cref{equ:_ineq} implies that
 \begin{align*}
 \| L_{\xi_f} L_{\eta_k}^* S L_{\xi'_{f'}} L_{\eta'_{k'}}^*|_{\widetilde{X}_h} \| \leq \| \xi_f \| \| \eta_k\| \| \xi'_{f'} \| \| \eta'_{k'}\| \| S |_{\widetilde{X}_{f' k'^{-1} h}}\| \ \leq \epsilon, \quad \forall h \geqslant \gamma_0.
 \end{align*}

Since $L$ is Nica covariant (see \cref{Nica_covariant_eq}), we have  $$I_{\TTT^r_X} = \overline{\rm{span}}\{ L_{\xi_f} L_{\eta_k}^* S L_{\xi'_{f'}}  L_{\eta'_{k'}}^* : S \in \cup_{z \in \ob(P)}  \overline{I}_{\TTT^r_X, z}, \xi_f,  \eta_k,  \xi'_{f'},  \eta'_{k'} \in X\}.$$ Then the claim above implies that $I_{\TTT^r_X} \cap L_{\CA_y} = \{0\}$.
\end{proof}

Next we will describe the universal $C^*$-algebras for  CNP-covariant representations of compactly aligned product systems over quasi-lattice ordered groupoids.

\begin{thm}
If $X$ is $\widetilde{\phi}$-injective, then there is a pair $(\NNN\OOO_X, j)$ such that
\begin{enumerate}
    \item $j$ is  an injective CNP-covariant Toeplitz representation of $X$ and $C^*(j) = \NNN\OOO_X$;
\item  $j$ is universal, i.e., for any  CNP-covariant representation $\psi = \{\psi_f\}_{f \in P(-, -)}$ of $X$,  there is a unique $*$-homomorphism $ \psi_* : \NNN\OOO_X \to C^*(\psi)$ such that the following diagram
     \begin{align*}
\xymatrix @R=0.2in {
   &  X \ar[ld]_{\psi} \ar[rd]^-{j} &  \\
  C^*(\psi) & &\NNN\OOO_X \ar[ll]_-{ \psi_*}
   }
\end{align*}
commutes.
\end{enumerate}
Moreover, $(\NNN\OOO_X, j)$ is the unique pair (up to canonical isomorphism) satisfying conditions (1) and (2).
\end{thm}

\begin{proof}
Let $(\NNN\TTT_X, i)$ be the Nica-Toeplitz algebra of $X$. 
Let $\NNN\OOO_X:= \NNN\TTT_X / I_{\NNN\TTT_X}$. By \cref{lem:CNP_con}, $j := q_{\NNN\TTT_X}\circ i$ is a CNP-covariant representation of $X$, where $q_{\NNN\TTT_X} : \NNN\TTT_X \to \NNN\OOO_X$ is the quotient map. 
Let $\psi = \{\psi_f\}$ be a CNP-covariant representation of $X$. By the universal property of $i$, there is a $*$-homomorphism $\psi' : \NNN\TTT_X \to C^*(\psi)$ such that $\psi = \psi'\circ i$. Since $\psi$ is Cuntz-Pimsner covariant, $I_{\NNN\TTT_X} \subset \ker(\psi')$.
Thus $\phi'$ induces a $*$-homomorphism $ \psi_* : \NNN\OOO_X \to C^*(\psi)$ such that $\psi =  \psi_* \circ j$. In particular, \cref{pro:_faith_ful_property} implies that $j$ is injective.
By the universal property of $j$, it is clear that $(\NNN\OOO_X, j)$ is the unique pair (up to canonical isomorphism) satisfying (1) and (2).
\end{proof}

The pair $(\NNN\OOO_X, j)$ is called the Cuntz-Nica-Pimsner algebra of $X$. These algebras generalize  Cuntz-Nica-Pimsner algebras of product systems over quasi-lattice ordered group (see \cite[ Proposition 3.12]{Sims2010}).

\section{ Cuntz-Nica-Pimsner algebras  are generalizations of full cross sectional $C^*$-algebras of Fell bundles}

In this section, we assume that $X= (\{\CA_x\}, \{X_f\} )$ is a compactly aligned product system over quasi-lattice ordered groupoid $(G, P)$. We will show that the full cross sectional $C^*$-algebras of Fell bundles of Morita equivalence bimodules are isomorphic to the related Cuntz-Nica-Pimsner algebras. In order to do this, we first provide a characterization of the Cuntz-Pimsner covariant representation.

\begin{lem}\label{lem:covariant_re}
Let $\psi = \{ \psi_f \}_{f \in P(-, -)}$ be a Toeplitz representation of $X= (\{\CA_x\}, \{X_f\} )$. If $\phi_f (\CA_{t(f)}) \subseteq \CK(X_f)$ and $\psi^{(f)}(\phi_f(-)) = \psi_{1_{t(f)}}(-)$ for every $f \in P(-, -)$, then $l^{g}_f(T_f) \in \CK(X_g)$ and
 \begin{align*}
 \psi^{(g)}(l^g_f(T_f)) = \psi^{(f)}(T_f),
 \end{align*}
for $f \leqslant g \in P(-, -)$, $T_f \in \CK(X_f)$.
\end{lem}

\begin{proof}
Let $f \in P(-, -)$. Since $\phi_f (\CA_{t(f)}) \subseteq \CK(X_f)$, we have $l^{g}_f(T_f) \in \CK(X_g)$ for all $g \geq f$ and $T_f \in \CK(X_f)$ by \cite[Proposition 4.7]{L95}.
Note that $$\psi^{(f^{-1}g )}(\phi_{f^{-1} g}(-)) = \psi_{1_{s(f)}}(-)$$ for $g \geq f$. For $\xi$, $\eta \in X_f$, $a, b \in  \CA_{s(f)}$, we have
\begin{align*}
\psi^{(f)}(\theta_{\xi a, \eta b}) &=\psi_f(\xi) \psi_{1_{s(f)}}(a b^*) \psi_f (\eta)^*\\
& = \psi_f(\xi) \psi^{(f^{-1} g)}( \phi_{f^{-1}g} (ab^*)) \psi_f (\eta)^*.
\end{align*}
By \cite[corollary 3.7]{MR1426840}, there eixst $\xi^n_i, \eta^n_i \in X_{f^{-1} g}$ such that
    \begin{align*}
        \phi_{f^{-1} g}(ab^*) = \lim_n \sum_i \theta_{\xi^n_i, \eta^n_i}, \mbox{ and }  l^{g}_f (\theta_{\xi a, \eta b} )= \lim_n \sum_i \theta_{\xi  \xi^n_i, \eta \eta^n_i}.
    \end{align*}
Therefore, $\psi^{(f)}(\theta_{\xi a, \eta b}) = \psi^{(g)}(l^g_f(\theta_{\xi a, \eta b}))$. Since $\CK(X_f)$ is a $C^*$-algebra generated by rank one compact operators and $X_f \CA_{s(f)} = X_f$, we have $$\psi^{(g)}(l^g_f(T_f)) = \psi^{(f)}(T_f)$$ for every $T_f \in \CK(X_f)$.
\end{proof}


\begin{prop}\label{prop:_CUntz_pimsner_covariant_relation}
Let $\psi : X \to \CB$ be a Toeplitz  representation of $X = (\{ \CA_x\}, \{ X_f\})$. Assume that any two elements in $P(-, y)$ has a least upper bound for every $y \in \ob(P)$, and the left actions $\phi_f : \CA_{t(f)} \to \CL(X_f)$ is injective for every $f \in P(-, -)$.

(1) If $\psi$  is Cuntz-Pimsner covariant, then $\psi^{(f)}(\phi_f(a)) = \psi_{1_{t(f)}}(a)$ for $a \in \phi_f^{-1}(\CK(X_f))$, $f\in P(-, -)$.

(2) If $\phi_f (\CA_{t(f)}) \subseteq \CK(X_f)$ and  $\psi^{(f)}(\phi_f(-)) = \psi_{1_{t(f)}}(-)$ for every $f\in P(-, -)$, then $\psi$ is  Cuntz-Pimsner covariant.
\end{prop}

\begin{proof}


Since $\phi_f$ is injective for $f \in P(-, -)$, we have $\widetilde{X}_f = X_f$ and $\widetilde{l}^g_f(-)= l^g_f(-)$ for all $ g \in P(-, -)$.

    (1)  Let $f \in P(-, y)$ and $a \in \phi^{-1}_f(\CK(X_f))$. For every $h \in P(-, y)$ we have
  \begin{align*}
l^{g}_{1_y}(a) -l^{g}_f(\phi_f(a)) = l^g_{f}(l^{f}_{1_{y}}(a) -\phi_f(a)) = 0
\end{align*}
for every $g \geqslant f \vee h$, since  $l^{f}_{1_y}(a) = \phi_f(a)$. Therefore
\begin{align*}
\psi_{1_y}(a) -\psi^{(f)}(\phi_f(a)) = \psi^{(1_y)}(a) -\psi^{(f)}(\phi_f(a)) = 0,
\end{align*}
since $\psi$ is Cuntz-Pimsner covariant.

(2) Suppose $l^g_h(T) = 0$ for $T \in \CL(X_h)$, $g \geqslant h \in P(-, -)$. Then for $\xi \in X_h$, $\eta \in X_{h^{-1} g}$,
    \begin{align*}
 0 & =  \rin{\CA_{s(g)}}[  l^g_h(T)(\xi  \eta)][  l^g_h(T)(\xi \eta)]\\
 & =  \rin{\CA_{s(g)}}[(T \xi ) \eta][ (T \xi)  \eta] \\
 & = \rin{\CA_{s(g)}}[ \eta][ \rin{\CA_{s(h)}}[T \xi][T \xi] \eta].
    \end{align*}
    Then $\phi_h( \rin{\CA_{s(h)}}[T \xi][T \xi] )= 0$. Since $\phi_h$ is injective, $T= 0$. Thus $l^g_h$ is injective.     Let $\F$ be a finite subset of $P(-, y)$, and $T_f \in \CK(X_f)$ such that
 \begin{align*}
 \sum_{f \in \F}l^g_f (T_f)=0
\end{align*}
 for large $g \in P(-, y)$.
Then
\begin{align*}
l^g_{h_0}(\sum_{f \in \F}l^{h_0}_f(T_f) ) = 0
\end{align*}
for $g \geqslant h_0 = \vee \F$. Since $l^g_{h_0}$ is injective, we have
     $\sum_{f \in \F}l^{h_0}_f(T_f) = 0$. Then by \cref{lem:covariant_re}, we have
\begin{align*}
\sum_{f \in \F} \psi^{(f)}(T_f) = \psi^{(h_0)}(\sum_{f \in \F}l^{h_0}_f(T_f))  =0.
\end{align*}
    Thus $\psi$ is  Cuntz-Pimsner covariant.
\end{proof}

Let $\psi = \{ \psi_f : X_f \to \CB(\CH)\}_{f \in P(-, -)}$ be a Toeplitz representation of $X$, where $\CB(\CH)$ is the $C^*$-algebra of all  bounded operators on a Hilbert space $\CH$.  Proposition 2.5 in \cite{L95} implies that the $*$-homomorphism $\psi^{(f)}$ admits a strict-SOT continuous extension $\psi^{(f)} : \CL(X_f) \to \CB(\CH)$. It is obvious that  $\psi : X \to \CB(\CH)$ is Nica covariant if and only if
\begin{align*}
   \psi^{(f)}(I)\psi^{(g)}(I) =  \psi^{(f \vee g)}(I)
\end{align*}
for  $f, g \in P(-, -)$.

\begin{cor}\label{cor :_NNNOOX_relation_OOOX}
Suppose that each pair in $P(-, y)$ has a least upper bound for all $y \in \ob(P)$ and  left actions $\phi_f : \CA_{t(f)} \to \CL(X_f)$ are all  injective  with $\phi_f (\CA_{t(f)}) \subseteq \CK(X_f)$ for $ f \in P(-, -)$. Let $\psi : X \to \CB(\CH)$ be a  Toeplitz representation of $X$ such that $$\psi_{1_x}(\CA_x) \psi_{1_y}(\CA_y) = 0$$ for all $x \neq y \in \ob(P)$. Then $\psi$ is CNP-covariant if and only if $$\psi^{(f)} \circ \phi_f (-)= \psi_{1_{t(f)}}(-).$$
\end{cor}

\begin{proof}
If  $\psi$ is CNP-covariant, then $\psi^{(f)} \circ \phi_f (-) = \psi_{1_{t(f)}}(-)$ by  \cref{prop:_CUntz_pimsner_covariant_relation} (1).
 Suppose $\psi^{(f)} \circ \phi_f(-) = \psi_{1_y}(-)$ for all $f \in P(x, y)$.  Let $$\psi_f(X_f) \CH := \overline{\rm{span}}\{ \psi_f(\xi_f) \eta: \xi_f \in X_f, \eta \in \CH
    \}.$$ Since $\phi_f (\CA_y) \subseteq \CK(X_f)$, we have  $\psi_{1_y} (\CA_y)\CH  \subseteq  \psi_f(X_f) \CH$ by \cite[Lemma 1.9]{FNMPR2003}. Note that $\CA_y\cdot X_f  = X_f$.
By the condition (2) in \cref{def:iso_rep_comp_alig_prod_system}, we have
   \begin{align*}
   \psi_f(X_f) \CH = \psi_f (\CA_y \cdot  X_f) \CH =\psi_{1_y}(\CA_y) \psi_f(X_f) \CH \subseteq \psi_{1_y}(\CA_y)\CH.
      \end{align*}
Therefore  $\psi_{1_y} (\CA_y)\CH =  \psi_f(X_f) \CH$ and $\psi^{(f)}(I) = \psi^{(1_y)}(I)$.
    Let $g \in P(-,-)$. Note $f \vee g < \infty$ if and only if $t(g) = y$. If $f \vee g < \infty $, then $$\psi^{(f)}(I)\psi^{(g)}(I) = \psi^{(1_y)}(I) = \psi^{(f\vee g)}(I).$$ If $t(g) \neq y$, then $\psi^{(f)}(I)\psi^{(g)}(I) = 0$ since $\psi_{1_y}(\CA_y) \psi_{1_t(g)}(\CA_{t(g)}) = 0$. Thus $\psi$ is Nica covariant. By \cref{prop:_CUntz_pimsner_covariant_relation} (2),  $\psi$ is CNP-covariant.
\end{proof}

Suppose $(G, P)$ is a quasi-lattice ordered groupoid such that each pair in $P(-, y)$ has a least upper bound for all $y \in \ob(P)$ and $G$ is generated by $P$. Let $\FB = \{ \FB_g\}_{g \in G(-, -)}$ be a Fell bundle over $G$ with each $\FB_g$ is a Morita equivalence bimodule.  Note that $\FB_g \FB_k = \FB_{gk}$ for $g, k \in G(-, -)$, $s(g) = t(k)$.  Then $$ \widehat{\FB} :=( \{\FB_x\}, \{ \FB_f\}_{f \in P(-, -)})$$ is a product system over $(G, P)$ with isomorphisms $\FB_f \otimes_{\FB_{s(f)}} \FB_h \cong \FB_{fh}$ corresponding to the multiplication in the Fell bundle $\FB$. It is obvious  that $\widehat{\FB}$ is a compactly aligned product system.

\begin{thm}
With notations as above. We have $C^*(\{ \FB_g\}_{g \in G(-, -)}) \cong \NNN\OOO_{\widehat{\FB}}$.
\end{thm}

\begin{proof}
Let $\pi^u = \{\pi^u_g\}_{g \in G(-, -)}$ be the universal representation of the Fell bundle $\FB$.
By the definition of  representations of Fell bundles,
we see that $\{ \pi^u_f : \FB_f \to C^*(\{ \FB_g\}_{g \in G(-, -)}) \}_{f \in P(-, -)}$ is a Toeplitz representation of the compactly aligned product system $ \widehat{\FB}$ in $C^*(\{ \FB_g\}_{g \in G(-, -)}) $.

For each $f \in P(-, -)$, we see that $\phi_f : \FB_{t(f)} \to \CK(\FB_f)$ is an isomorphism which sends $\lin{\FB_{t(f)}}[\xi_f][\eta_f]$ to $\theta_{\xi_f, \eta_f}$ for $ \xi_f, \eta_f \in \FB_f$.
 Then by the definition of representations of the Fell bundles, we have
\begin{align*}
(\pi^{u})^{(f)} (\phi_f (\sum_i \lin{\FB_{t(f)}}[\xi_i][\eta_i]))
& = (\pi^{u})^{(f)} (\sum_i \theta_{\xi_i, \eta_i})\\
& = \sum_i \pi^u_f(\xi_i) \pi^u_f(\eta_i)^*\\
& = \sum_i  \pi^u_{1_{t(f)}}(\xi_i \eta^*_i)\\
& = \pi^u_{1_{t(f)}}(\sum_i \lin{\FB_{t(f)}}[\xi_i][\eta_i]),
\end{align*}
where $\sum_i \lin{\FB_{t(f)}}[\xi_i][\eta_i]$ is a finite sum of  $\FB_{t(f)}$ with $\xi_i, \eta_i \in \FB_f$.
Since $\lin{\FB_{t(f)}}[\FB_f][\FB_f]$ is dense in $\FB_{t(f)}$, we have $$(\pi^{u})^{(f)}\circ \phi_f(-) = \pi^u_{1_{t(f)}}(-).$$
It is easy to see that
\begin{align*}
\pi^u_{1_x}(\FB_x) \pi^u_{1_y}(\FB_y) =0
\end{align*}
 for $x \neq y \in \ob(P)$.
By \cref{cor :_NNNOOX_relation_OOOX}, $\{ \pi^u_f\}_{f \in P(-, -)}$ is a CNP-covariant Toeplitz representation of $\widehat{\FB}$. Recall that $(\NNN\OOO_
{\widehat{\FB}}, j)$ is the Cuntz-Nica-Pimsner algebra of $\widehat{\FB}$. By the universal property of $j$, there is a $*$-homomorphism $\phi_1$ such that $\phi_1 \circ j_f = \pi^u_f$ for $f \in P(-, -)$.

Since $G$ is generated by $P$, for each $g \in G(-, -)$ there exist $f, h \in P(-, -)$  such that $\FB_g = \FB_f \FB_h^*$. Let $\pi_g : \FB_g \to \NNN\OOO_{\widehat{\FB}} $ be a linear map defined by
$$\pi_g(\xi_f \eta_h^*) := j_f(\xi_f) j_h(\eta_h)^*$$ for $\xi_f \in \FB_f$, $\eta_h \in \FB_h$, $f, h \in P(-, -)$ with $fh^{-1} = g$.  We see that  $$\pi_g(\xi_f \eta_h^*)^* = j_h(\eta_h) j_f(\xi_f)^* = \pi_{g^{-1}}(\eta_h \xi_f^*).$$ Then $\pi_g(\xi)^* = \pi_{g^{-1}}(\xi^*)$ for $\xi \in \FB_g$.

If $f \leq h \in P(-, -)$, since $j$ is a Toeplitz representation of $\widehat{\FB}$, we have
\begin{align*}
j_f(\xi_f)^* j_h(\eta_f \zeta_{f^{-1}h}) &= j_f(\xi_f)^* j_f(\eta_f) j_{f^{-1}h}(\zeta_{f^{-1}h}) \\
&= j_{1_{s(f)}}(\rin{\FB_{s(f)}}[\xi_f][\eta_f]) j_{f^{-1}h}(\zeta_{f^{-1}h}) \\
&= j_{f^{-1}h}(\xi_f^* \eta_f \zeta_{f^{-1}h})
\end{align*}
for all $\xi_f, \eta_f \in \FB_f$, $\zeta_{f^{-1}h} \in \FB_{f^{-1}h}$. It follows that $j_f(\xi_f)^* j_h(\eta_h)= j_{f^{-1}h}(\xi_f^* \eta_h)$ for  $\eta_{h} \in \FB_h$.
For $g, k \in G(-, -)$ with $s(g) = t(k)$, we have
\begin{align*}
&\pi_g(\xi_{f_1} \eta^*_{h_1}) \pi_k(\lin{\FB_{t(k)}}[\beta_{h_1\vee f_2}][\zeta_{h_1\vee f_2}]\xi'_{f_2} \eta'^{*}_{h_2})\\
& = j_{f_1}(\xi_{f_1}) j_{h_1}(\eta_{h_1})^* j_{1_{t(k)}} (\lin{t(k)}[\beta_{h_1\vee f_2}][\zeta_{h_1\vee f_2}])j_{f_2} (\xi'_{f_2}) j_{h_2}(\eta'_{h_2})^*\\
& =  j_{f_1}(\xi_{f_1}) j_{h_1}(\eta_{h_1})^*j_{h_1\vee f_2}(\beta_{h_1\vee f_2}) j_{h_1\vee f_2}(\zeta_{h_1\vee f_2})^*j_{f_2} (\xi'_{f_2}) j_{h_2}(\eta'_{h_2})^*\\
& = j_{f_1h^{-1}_1 (h_1\vee f_2)}(\xi_{f_1} \eta^*_{h_1} \beta_{h_1\vee f_2}) j_{h_2f^{-1}_2 (h_1\vee f_2)}((\eta'_{h_2}\xi'^*_{f_2} \zeta_{h_1\vee f_2})^*\\
& = \pi_{gk}(\xi_{f_1} \eta^*_{h_1} \beta_{h_1\vee f_2}  \zeta_{h_1\vee f_2}^*  \xi'_{f_2} \eta'^{*}_{h_2})
\end{align*}
where $f_1, h_1, f_2 , h_2 \in P(-, -)$ with $f_1h^{-1}_1 = g$, $f_2h^{-1}_2 = k$, $\xi_{f_1} \in \FB_{f_1}$, $\eta_{h_1} \in \FB_{h_1}$, $\beta_{h_1\vee f_2}, \zeta_{h_1\vee f_2} \in \FB_{h_1\vee f_2}$, $\xi'_{f_2} \in \FB_{f_2}$, $\eta'_{h_2} \in \FB_{h_2}$. Since $\lin{\FB_{t(k)}}[\FB_{h_1\vee f_2}][\FB_{h_1\vee f_2}] \FB_{k}$ is dense in $\FB_{k}$, we have $\pi_g(-) \pi_k(-) = \pi_{gk}(-)$.
If $s(g) \neq t(k)$, by condition (1) in \cref{rem:_some_p_of_N}, we have $\pi_g(-) \pi_k(-)  = 0$ since $j$ is a Nica covariant Toeplitz representation of $\widehat{\FB}$.  Therefore, $\pi = \{ \pi_g\}_{g \in G(-, -)}$ is a representation of the Fell bundle $\FB$.
Then  there is a $*$-homomorphism $$\phi_2 : C^*(\{ \FB_g\}_{g \in G(-, -)}) \to \NNN\OOO_{\widehat{\FB}}$$ such that $\phi_2 \circ \pi^u_g(-) = \pi_g(-)$.
It is obvious that $\phi_1$ and $\phi_2$ are mutually reversible. Then $$C^*(\{\FB_g\}_{g \in G(-, -)}) \cong \NNN\OOO_{\widehat{\FB}}.$$
This completes the proof. 
\end{proof}

\bibliographystyle{amsplain}

\end{document}